\documentclass[11pt]{amsart}
\usepackage{amscd}
\usepackage{amsfonts}
\usepackage{amssymb,latexsym}
\usepackage{enumerate}
\usepackage{amsmath}
\usepackage{mathrsfs}
\usepackage{amsthm}
\usepackage{hyperref}
\usepackage[square, comma, sort&compress, numbers]{natbib}
\usepackage{bm}
\usepackage{esint}
\usepackage{fancyhdr}
\usepackage{lastpage}
\usepackage{layout}
\usepackage{ulem}
\usepackage{extarrows}
\usepackage{geometry}
\newcommand{\ol}{\overline}

\newcommand{\pa}{\partial}

\DeclareMathOperator\tr{tr}
\DeclareMathOperator\id{id}
\DeclareMathOperator\rank{rank}
\DeclareMathOperator\Ric{Ric}

\DeclareMathOperator\dvol{dvol}
\DeclareMathOperator\End{End}

\DeclareMathOperator\im{im}

\DeclareMathOperator\Aut{Aut}
\DeclareMathOperator\Unitary{U}
\DeclareMathOperator\Ad{Ad}
\DeclareMathOperator\Hom{Hom}
\DeclareMathOperator\Div{div}
\makeatletter
\begin{document}
\newcounter{remark}
\newcounter{theor}
\setcounter{remark}{0}
\setcounter{theor}{1}
\newtheorem{claim}{Claim}[section]
\newtheorem{theorem}{Theorem}[section]
\newtheorem{lemma}{Lemma}[section]
\newtheorem{corollary}{Corollary}[section]
\newtheorem{corollarys}{Corollary}
\newtheorem{proposition}{Proposition}[section]
\newtheorem{question}{Question}[section]
\newtheorem{defn}{Definition}[section]
\newtheorem{examp}{Example}[section]
\newtheorem{assumption}{Assumption}[section]
\newtheorem{rem}{Remark}[section]
\newtheorem*{theorem1}{Theorem}
\newtheorem*{maintheorem}{Main Theorem}
\newtheorem*{maintheorem2}{Main Theorem \uppercase\expandafter{\romannumeral2}}
\numberwithin{equation}{section}
\title{Harmonic metrics and semi-simpleness}
\author{Di Wu}
\address{Di Wu, School of Mathematics and Statistics, Nanjing University of Science and Technology, Nanjing 210094, People's Republic of China}
\email{wudi123@mail.ustc.edu.cn}
\author{Xi Zhang}
\address{Xi Zhang, School of Mathematics and Statistics, Nanjing University of Science and Technology, Nanjing 210094, People's Republic of China}
\email{mathzx@ustc.edu.cn}
\thanks{The research was supported by the National Key R\&D Program of China 2020YFA0713100. Both authors are partially supported by NSF in China No.12141104. The first author is also supported by the Jiangsu Funding Program for Excellent Postdoctoral Talent 2022ZB282}

\maketitle
\begin{abstract}
Given a flat vector bundle over a compact Riemannian manifold, Corlette \cite{Co1988} and Donaldson \cite{Do1987b} proved that it admits harmonic metrics if and only if it is semi-simple. In this paper, we extend this equivalence to arbitrary vector bundles without any additional hypothesis, the result can be viewed as a Riemannian Hitchin-Kobayashi correspondence. Furthermore, we also prove an equivalence of categories in Sasakian geometry, relating the category of projective flat complex vector bundles to the category of Higgs bundles.
\end{abstract}
\section{Introduction}
A unifying principle in geometric analysis predicts that existences of canonical metrics should be closely related to appropriate algebraic stability conditions on underlying geometric objects, one may refer to the works on Yau-Tian-Donaldson conjecture for extremal metrics in K\"{a}hler geometry and Donaldson-Uhlenbeck-Yau characterization(also referred as Hitchin-Kobayashi correspondence or Kobayashi-Hitchin correspondence) for Hermitian-Einstein metrics in holomorphic vector bundles. It is our aim in present paper to investigate the issue on general vector bundles. Unless indicated explicitly otherwise, vector bundles could be real or complex, whose metrics are Riemannian or Hermitian respectively.
\par Suppose that $E$ is a smooth vector bundle over a compact Riemannian manifold $(M,g)$ and $\nabla$ is a connection on $E$. We call $(E,\nabla)$ semi-simple(also called completely reducible or reductive in the literature) if it splits as a direct sum of simple sub-bundles, where simple(also called irreducible or stable in the literature) means that there exists no nontrivial $\nabla$-invariant sub-bundle. A vector bundle $E$ is said to be semi-simple if $(E,\nabla)$ is semi-simple for some $\nabla$. Naturally there arises a fundamental problem: \textit{Does there exist certain best canonical metrics on $(E,\nabla)$?} On the other hand, motivated by calculating the Euler characteristic number $\chi(E)$ via the Gauss-Bonnet-Chern formula and its ramifications, one may ask: \textit{How to find metric compatible connections on $E$?} For this purpose, we write for a metric $K$ on $E$,
\begin{equation}\begin{split}
\nabla=\nabla_K+\psi_{K},
\end{split}\end{equation}
where $\nabla_K$ is a connection preserving $K$ and $\psi_{K}$ is a $\End(E)$-valued $1$-form. We define
\begin{equation}\begin{split}
\mathcal{E}_\nabla(K)=\frac{1}{2}\int_M|\psi_{K}|^2\dvol_g,
\end{split}\end{equation}
and the point is to minimize $\mathcal{E}_\nabla(K)$ when $K$ varies.
\par The canonical metric concerned in this paper is the critical point of $\mathcal{E}_\nabla$ which means
\begin{equation}\begin{split}\label{harmonicmetricequation}
\nabla_H^\ast\psi_{H}=0,
\end{split}\end{equation}
and we then call $H$ a harmonic metric. Another reason for this candidate lies in the scenario of nonabelian Hodge theory in K\"{a}hler geometry \cite{Si1992}, in which $\nabla$ is assumed to be flat and the existence of harmonic metrics leads to the existence of Higgs structures. It is remarked that the flatness means any metric $H$ corresponds to an equivariant map
\begin{equation}\begin{split}
f_H: \tilde{M}\rightarrow GL(r)/U(r),
\end{split}\end{equation}
where $\tilde{M}$ is the universal covering space of $M$ and $r=\rank(E)$, Moreover, it is known that $H$ being harmonic iff $f_H$ being harmonic. So in general case, lacking of flatness may block off methods used in equivariant harmonic maps. Detecting harmonic metrics on vector bundles is a nonlinear system generalization of solving Laplace equations and obstructions would appear.
\par In this paper, we prove(see Theorem \ref{semisimpleharmonic})
\begin{theorem}\label{thm1}
Let $(E,\nabla)$ be a vector bundle over a compact Riemannian manifold with an arbitrary connection $\nabla$, then it admits harmonic metrics if and only if it is semi-simple.
\end{theorem}
\par Note that Theorem \ref{thm1} is a Riemannian counterpart of the Hitchin-Kobayashi correspondence in K\"{a}hler geometry \cite{Do1985,Do1987a,UY1986}. Concerning Theorem \ref{thm1}, there are a few previous works in the presence of extra assumptions which we should mention. As a postscript to Hitchin's paper \cite{Hi1987} on self-duality equations, borrowing ideas from harmonic maps, Donaldson \cite{Do1987b} proved that any vector bundle over a compact Riemannian surface equipped with a semi-simple flat connection must admit harmonic metrics. Deforming a semi-simple flat connection via an evolution equation and utilizing Uhlenbeck compactness \cite{Uh1982}, it is Corlette who proved Theorem \ref{thm1} for flat vector bundles in \cite{Co1988}. Many works are constantly engaged in generalizations and applications on Corlette-Donaldson's tremendous criterion since then, see \cite{CJY2019,Co1992,JZ1996,JZ1997,KV1998,Li1996,Lo2009,Lu1999,Mo2009,PZZ2019,Si1990,WZ2021b}. Among other things, if $(E,\nabla)$ is non-Hermitian Yang-Mills(see \cite{KV1998} for precise definition) instead of flatness and the base space is a compact K\"{a}hler manifold, Kaledin-Verbitsky conjectured(\cite[Conjecture 8.7]{KV1998}) that the existence of harmonic metrics is equivalent to a kind of stability condition, see \cite{PSZ2022}.
\par Given an automorphism $\sigma\in\Aut(E)$, it acts on a connection $\nabla$ by $\nabla\mapsto\sigma\circ \nabla\circ\sigma^{-1}$ and the space $\mathcal{M}_S$ of isomorphism classes of simple connections is realized as the quotient
\begin{equation}\begin{split}
\{Simple\ connections\ on\ E\}/\Aut(E).
\end{split}\end{equation}
We say a connection $\nabla$ irreducible if there is no nontrivial $\nabla$-invariant endomorphism and the space $\mathcal{M}_{I,K}$ of isomorphism classes of irreducible connections with $\nabla_K^\ast\psi_K=0$ is defined by
\begin{equation}\begin{split}
\{Irreducible\ connections\ \nabla\ on\ E\ with\ \nabla_K^\ast\psi_K=0\}/\Unitary(E,K),
\end{split}\end{equation}
where $\Unitary(E,K)\subseteq\Aut(E)$ is the sub-group preserving $K$.
\par Using Theorem \ref{thm1}, we have(see Corollary \ref{correspondence})
\begin{corollary}\label{cor1}
Assume $(E,K)$ is a Hermitian vector bundle over a compact Riemannian manifold, then there is a bijection from $\mathcal{M}_S$ to $\mathcal{M}_{I,K}$.
\end{corollary}
\par Next we are concerned with vector bundles which have Sasakian manifolds as base spaces. Recall Sasakian manifolds were first introduced in \cite{Sa1960}, which is a subject lying on the intersection of contact, CR, Riemannian and K\"{a}hler geometry. A standard model of a compact Sasakian manifold is the odd dimensional unit sphere $S^{2n+1}$. With the robust discovery of its relevance in string theory and anti-de Sitter space duality conjecture/conformal field theory in mathematical physics \cite{Ma1998}, there has been a considerable increase of interests in Sasakian manifolds of late, see \cite{BG2008,GMSW2004,MS2006,MSY2008} and references therein.
\par An application of Theorem \ref{thm1} emerges in Sasakian geometry, we prove(see Theorem \ref{Sasakiancorrespondence})
\begin{theorem}\label{thm3}
For a compact Sasakian manifold $(M,T_{1,0}M,\eta)$ of dimension $2n+1$, there exists an equivalence between the category of rank $r$ semi-simple basic projective flat complex vector bundles over $M$ and the category of rank $r$ poly-stable basic Higgs bundles over $M$ with
\begin{equation}\begin{split}
\int_M\left(c_{2,B_{\mathcal{F}_\xi}}(E)-\frac{r-1}{2r}c_{1,B_{\mathcal{F}_\xi}}^2(E)\right)\wedge(d\eta)^{n-2}\wedge\eta=0,
\end{split}\end{equation}
where $\xi$ is the characteristic direction and $\mathcal{F}_\xi$ is the associated foliation.
\end{theorem}
The existence of harmonic metrics is employed to produce the functor that from semi-simple basic projective flat complex vector bundles to basic Higgs bundles. Beyond that, it requires to prove the key property $\nabla_{H}(\psi_H(\xi))=0$ for harmonic metric $H$, known by Biswas-Kasuya \cite{BK2021a} for flat vector bundles via modifying the spinorial proof of \cite[Theorem 4.1]{Pe2002}. Under the weaker condition $\nabla_{H,\xi}\nabla_H^\ast\psi_H=0$, we shall give a simple maximum principle proof without involving special curvature theory of Tanaka-Webster connections on Sasakian manifolds and the commutative formula \cite[Corollary 2.1]{Pe2002} concerning the Dirac operator, see Lemma \ref{basicharmonic1}. The opposite direction is mainly inspired by \cite{Si1988} and Baraglia-Hekmati's work \cite{BH2022} on a foliated Hitchin-Kobayashi correspondence, where the notion of stability was addressed.
\par Our Theorem \ref{thm3} extends the main result in \cite{BK2021a}(see also \cite{BM2018} for the case of quasi-regular Sasakian manifolds) to the nonvanishing Chern classes case. Note these are Sasakian analogues of the Corlette-Donaldson-Hitchin-Simpson correspondence built upon \cite{Co1988,Do1987b,Hi1987,Si1988,Si1992}, dating back to Narasimhan-Seshadri \cite{NS1965}, Donadlson \cite{Do1985,Do1987a} and Uhlenbeck-Yau \cite{UY1986}. Readers are also refereed to \cite{BK2021b,Ka2020,KM2022} for recent works concerning related topics.
\section{Harmonic metrics on vector bundles}
\subsection{Preliminaries}
We assume $E$ is a vector bundle over a compact Riemannian manifold $(M,g)$ and set $A^{p}(M,E)=\Gamma(\Lambda^pT^\ast M\otimes E)$, the space of $p$-forms on $M$ with values in $E$. Associated to a connection $\nabla$ on $E$, the exterior differential operator
\begin{equation}\begin{split}
D: A^{p}(M,E)\rightarrow A^{p+1}(M,E)
\end{split}\end{equation}
is defined by requiring $D(\alpha\otimes u)=d\alpha\otimes u+(-1)^p\alpha\wedge\nabla u$ for any $\alpha\in A^p(M)$ and $u\in\Gamma(E)$. Then the curvature is given by $F_\nabla=D\circ\nabla\in A^2(M,\End(E))$. For $\omega\in A^p(M,E)$, we have
\begin{equation}\begin{split}\label{exteriord}
D\omega(e_0,...,e_p)
&=\sum\limits_{k=0}^p(-1)^{k}\nabla_{e_k}(\omega(e_0,e_1,...,\hat{e_k},...,e_p))
\\&+\sum\limits_{0\leq k<l\leq p}(-1)^{k+l}\omega([e_k,e_l],e_1,...,\hat{e_k},...,\hat{e_l},...,e_p),
\end{split}\end{equation}
where $e_1,...,e_p\in\Gamma(TM)$ and symbols covered by $\wedge$ are omitted. If the tangent bundle $TM$ is equipped with a connection $\nabla^{TM}$, the action of $\nabla$ may be further extended to tensorial combinations of $TM$ and $E$ as well as their duals. Then it also holds
\begin{equation}\begin{split}
D\omega(e_0,...,e_p)
&=\sum\limits_{k=0}^p(-1)^{k}(\nabla_{e_k}\omega)(e_0,...,\hat{e_k},...,e_p)
\\&-\sum\limits_{0\leq k<l\leq p}(-1)^{k+l}\omega(T_{\nabla^{TM}}(e_k,e_l),e_1,...,\hat{e_k},...,\hat{e_l},...,e_p),
\end{split}\end{equation}
where $T_{\nabla^{TM}}$ denotes the torsion. For a metric $K$ on $E$, we have the pointwise inner-product
\begin{equation}\begin{split}
(\bullet,\bullet)_K: A^p(M,E)\times A^p(M,E)\rightarrow C^\infty(M),
\end{split}\end{equation}
\begin{equation}\begin{split}
(\omega,\theta)\mapsto\sum\limits_{i_1<...<i_p}K(\omega(e_{i_1},...,e_{i_p}),\theta(e_{i_1},...,e_{i_p})),
\end{split}\end{equation}
where $\{e_i\}_{i=1}^{\dim M}$ is an orthogonal unit basis of $TM$. Henceforth, we will omit the subscript $K$ if there is no ambiguity. The generalized inner-product
\begin{equation}\begin{split}
K(\bullet,\bullet): A^p(M,E)\times A^q(M,E)\rightarrow A^{p+q}(M)
\end{split}\end{equation}
is defined by requiring $K(\alpha\otimes u,\beta\otimes v)=(u,v)\alpha\wedge\beta$ for any $\alpha\in A^p(M)$, $\beta\in A^q(M)$ and $u,v\in\Gamma(E)$. In particular, for $\omega,\theta\in A^p(M,E)$ we have $K(\omega,\ast\theta)=(\omega,\theta)\dvol_g$, where $\ast$ is the Hodge star operator acting on the form component.
\par Given above $\nabla$ and $K$, we define $\psi_K\in A^1(M,\End(E))$ by
\begin{equation}\begin{split}\label{metricdecom}
K(\psi_{K}(u),v)=\frac{1}{2}\left(K(\nabla u,v)+K(u,\nabla v)-dK(u,v)\right),
\end{split}\end{equation}
where $u,v\in\Gamma(E)$. One can easily check $\psi_K$ is self-adjoint and
\begin{equation}\begin{split}
\nabla_K=\nabla-\psi_K: A^{0}(M,E)\rightarrow A^{1}(M,E)
\end{split}\end{equation}
is a connection preserving $K$. The co-differential operator
\begin{equation}\begin{split}
D_K^\ast: A^{p}(M,E)\rightarrow A^{p-1}(M,E)
\end{split}\end{equation}
is determined by $D_K$, the exterior differential operator of $\nabla_K$,
\begin{equation}\begin{split}
\int_M(D_K\omega,\theta)\dvol_g=\int_M(\omega,D_K^\ast\theta)\dvol_g,
\end{split}\end{equation}
where $\omega\in A^{p-1}(M,E),\theta\in A^p(M,E)$. Then it holds
\begin{equation}\begin{split}
D_K^\ast\theta(e_1,..,e_{p-1})=-\tr_g\nabla_{K,\bullet}\theta(\bullet,e_1,...,e_{p-1}),
\end{split}\end{equation}
where for notational simplicity, $\nabla_K$ also denotes the connection on $\Lambda^pT^\ast M\otimes E$ which acts on $\Lambda^pT^\ast M$ by the Levi-Civita connection $\nabla^g$.
\par For two metrics $K$ and $H$ on $E$, we denote by $f=K^{-1}H$ the endormorphism given by
\begin{equation}\begin{split}
H(\bullet,\bullet)=K(f(\bullet),\bullet).
\end{split}\end{equation}
By $(\ref{metricdecom})$ and the fact that $f$ is self-adjoint with respect to $K$ and $H$, we have
\begin{equation}\begin{split}
H(\psi_{H}(u),v)
&=\frac{1}{2}\left(H(\nabla u,v)+H(u,\nabla v)-dH(u,v)\right)
\\&=\frac{1}{2}\left(K((f\circ \nabla )u,v)+K(f(u),\nabla v)-dK(f(u),v)\right)
\\&=\frac{1}{2}\left(K((\nabla\circ f-\nabla f)u,v)+K(f(u),\nabla v)-dK(f(u),v)\right)
\\&=K((\psi_{K}\circ f)u,v)-\frac{1}{2}K((\nabla f)u,v)
\\&=H((f^{-1}\circ\psi_{K}\circ f)u,v)-\frac{1}{2}H((f^{-1}\nabla f)u,v)
\\&=H(\psi_{K}(u),v)+H((f^{-1}[\psi_K,f])u,v)
-\frac{1}{2}H((f^{-1}\nabla f)u,v)
\\&=H(\psi_{K}(u),v)-\frac{1}{2}H((f^{-1}\delta_{K}f)u,v),
\end{split}\end{equation}
where $\delta_{K}=\nabla_K-\psi_K$. It follows
\begin{equation}\begin{split}
\nabla_H=\nabla_K+\frac{1}{2}f^{-1}\delta_{K}f,\ \psi_{H}=\psi_{K}-\frac{1}{2}f^{-1}\delta_{K}f.
\end{split}\end{equation}
According to the definition, we deduce
\begin{equation}\begin{split}\label{DHKstarpsiHK}
\nabla_H^\ast\psi_{H}&=-\tr_g\nabla_{H}\psi_{H}
\\&=-\tr_g(\nabla_K+\frac{1}{2}f^{-1}\delta_{K}f)(\psi_{K}-\frac{1}{2}f^{-1}\delta_{K}f)
\\&=-\tr_g\nabla_{K}\psi_{K}
+\frac{1}{2}\tr_g\nabla_{K}(f^{-1}\delta_{K}f)-\frac{1}{2}\tr_g[f^{-1}\delta_{K}f,\psi_{K}]
\\&=\nabla_K^\ast\psi_{K}-\frac{1}{2}\nabla_K^\ast(f^{-1}\delta_{K}f)+\frac{1}{2}\tr_g[\psi_{K},f^{-1}\delta_{K}f]
\\&=\nabla_K^\ast\psi_K+\frac{1}{2}\tr_g\nabla(f^{-1}\delta_{K}f).
\end{split}\end{equation}
\par Set $f=h^{\ast K}h$ and $H(\bullet,\bullet)=K(h(\bullet),h(\bullet))$. Consider the gauge action $h(\bullet)=h\circ\bullet\circ h^{-1}$ and write $h(\nabla)=h(\nabla)_K+\psi_{h(\nabla),K}$ for the decomposition of $h(\nabla)$ in term of $K$, then
\begin{equation}\begin{split}
dK(u,v)
&=dH(h^{-1}(u),h^{-1}(v))
\\&=H((\nabla_H\circ h^{-1})u,h^{-1}(v))
+H(h^{-1}(u),(\nabla_H\circ h^{-1})v)
\\&=K(h(\nabla_H)u,v)+K(u,h(\nabla_H)v),
\end{split}\end{equation}
\begin{equation}\begin{split}
K(\psi_{h(\nabla),K}(u),v)
&=\frac{1}{2}\left(K(h(\nabla)u,v)+K(u,h(\nabla)v)-dK(u,v)\right)
\\&=\frac{1}{2}H((\nabla\circ h^{-1})u,h^{-1}(v))+\frac{1}{2}H(h^{-1}u,(\nabla\circ h^{-1})v)
\\&-\frac{1}{2}dH(h^{-1}(u),h^{-1}(v))
\\&=H((\psi_{H}\circ h^{-1})u,h^{-1}(v))
\\&=K(h(\psi_{H})u,v).
\end{split}\end{equation}
Namely, we have $h(\nabla)_K=h(\nabla_H)$, $\psi_{h(\nabla),K}=h(\psi_{\nabla,H})$, and hence it holds
\begin{equation}\begin{split}
h(\nabla)^\ast_K\psi_{h(\nabla),K}
&=-\tr_gh(\nabla)_{K}\psi_{h(\nabla),K}
\\&=-\tr_gh(\nabla_H)h(\psi_{H})
\\&=-h(\tr_g\nabla_{H}\psi_{H})
\\&=h(\nabla_H^\ast\psi_{H}).
\end{split}\end{equation}
Moreover, by $(\ref{exteriord})$ we also know $h(D_H)$ is the exterior differential operator relative to $h(\nabla)_K$.
\subsection{Harmonic metrics}
We shall apply the method of continuity to detect harmonic metrics on a simple vector bundle $(E,\nabla)$ over a compact Riemannian manifold $(M.g)$. Let us fix a background metric $K$ and consider the following family of equations:
\begin{equation}\begin{split}\label{continuitypath}
\nabla_{H_\epsilon}^\ast\psi_{H_\epsilon}=\epsilon\log f_\epsilon,\ f_\epsilon=K^{-1}H_\epsilon.
\end{split}\end{equation}
Pick a metric $K_0$ and a function $u$ with $\Delta_gu=-2\rank(E)^{-1}\tr\nabla_{K_0}^\ast \psi_{K_0}$, where $\Delta_g$ denotes the Beltrami-Laplace. Define $\hat{K}=e^uK_0$, $K=\hat{K}\exp(-\nabla_{\hat{K}}^\ast\psi_{\hat{K}})$, and then by $(\ref{DHKstarpsiHK})$, we obtain
\begin{equation}\begin{split}\label{initial}
\tr\nabla_K^\ast\psi_K=\tr\nabla_{\hat{K}}^\ast\psi_{\hat{K}}=0
\end{split}\end{equation}
Moreover, one easily concludes $f_1=\exp\nabla_{\hat{K}}^\ast\psi_{\hat{K}}$ must solve $(\ref{continuitypath})$ with $\epsilon=1$.
\par We denote by $I$ the set of $\epsilon\in[0,1]$ such that $(\ref{continuitypath})$ is solvable and thus $1\in J$. For the openness of $I$, we consider the operator
\begin{equation}\begin{split}
\hat{L}:(0,1]\times L_k^p(S^{+}_K)\rightarrow L^p_{k-2}(S_K),\ (\epsilon,f)\mapsto f\circ(\nabla_{Kf}^\ast\psi_{Kf}-\epsilon\log f),
\end{split}\end{equation}
where $L_k^p(S_K)$($L_k^p(S^{+}_K)$) is the Sobolev space of (positive)self-adjoint sections of $\End(E)$. Set
\begin{equation}\begin{split}
\nabla^f=\Ad f^{\frac{1}{2}}\circ\nabla\circ\Ad f^{-\frac{1}{2}},\ \delta_K^f=\Ad f^{-\frac{1}{2}}\circ\delta_K\circ\Ad f^{\frac{1}{2}},
\end{split}\end{equation}
where $\Ad f^{\frac{1}{2}}(\bullet)=f^{\frac{1}{2}}\circ\bullet\circ f^{-\frac{1}{2}}$ and $\Ad f^{-\frac{1}{2}}$ is defined accordingly. In the normal coordinate,
\begin{equation}\begin{split}\label{open1}
\tr_g\nabla^f\delta_{K}^f(f^{-\frac{1}{2}}\delta ff^{-\frac{1}{2}})
&=\tr_g\nabla^f(f^{-\frac{1}{2}}\delta_{K}(\delta ff^{-1})f^{\frac{1}{2}})
\\&=\tr_g\Ad f^{\frac{1}{2}}(\nabla(f^{-1}\delta_{K}(\delta ff^{-1})f))
\\&=\tr_g\Ad f^{\frac{1}{2}}(\nabla(f^{-1}\delta_{K}\delta f)+\nabla(f^{-1}\delta f\delta_{K}f^{-1}f))
\\&=\tr_g\Ad f^{\frac{1}{2}}(\nabla(f^{-1}\delta_{K}\delta f)+\nabla(\delta f^{-1}\delta_{K}f))
\\&=2\Ad f^{\frac{1}{2}}(\delta\nabla_{Kf}^\ast\psi_{Kf}).
\end{split}\end{equation}
It follows for any $(\epsilon,f_\epsilon)$ with $\nabla_{Kf_\epsilon}^\ast\psi_{Kf_\epsilon}-\epsilon\log f_\epsilon=0$,
\begin{equation}\begin{split}\label{open2}
\delta_2\hat{L}(\epsilon,f_\epsilon)&=\delta f_\epsilon(\nabla_{Kf_\epsilon}^\ast\psi_{Kf_\epsilon}-\epsilon\log f_\epsilon)+f\delta(\nabla_{Kf}^\ast\psi_{Kf}-\epsilon\log f)
\\&=f\delta\nabla_{Kf_\epsilon}^\ast\psi_{Kf_\epsilon}-\epsilon f_\epsilon\delta\log f_\epsilon
\\&=\frac{1}{2}f^{\frac{1}{2}}_\epsilon\tr_g\nabla^{f_\epsilon}\delta_{K}^{f_\epsilon}(f^{-\frac{1}{2}}_\epsilon\delta f_\epsilon f^{-\frac{1}{2}}_\epsilon)f^{\frac{1}{2}}_\epsilon-\epsilon f_\epsilon\delta\log f_\epsilon.
\end{split}\end{equation}
On the other hand, since $d(u,v)=(\nabla^f u,v)+(u,\delta^f_Kv)$ for $u,v\in\Gamma(\End(E))$, we have for any $\phi\in S_K$(where $S_K$ is the space of self-adjoint sections of $\End(E)$) with $\delta_2\hat{L}(\epsilon,f_\epsilon)(\phi)=0$,
\begin{equation}\begin{split}\label{open3}
\Delta_g|f^{-\frac{1}{2}}_\epsilon\phi f^{-\frac{1}{2}}_\epsilon|^2
&=(\tr_g\nabla^{f_\epsilon}\delta_{K}^{f_\epsilon}(f^{-\frac{1}{2}}_\epsilon\phi f^{-\frac{1}{2}}_\epsilon),f^{-\frac{1}{2}}_\epsilon\phi f^{-\frac{1}{2}}_\epsilon)
\\&+(f^{-\frac{1}{2}}_\epsilon\phi f^{-\frac{1}{2}}_\epsilon,\tr_g\delta_{K}^{f_\epsilon}\nabla^{f_\epsilon}(f^{-\frac{1}{2}}_\epsilon\phi f^{-\frac{1}{2}}_\epsilon))
\\&+|\nabla^{f_\epsilon}(f^{-\frac{1}{2}}_\epsilon\phi f^{-\frac{1}{2}}_\epsilon)|^2+|\delta_K^{f_\epsilon}(f^{-\frac{1}{2}}_\epsilon\phi f^{-\frac{1}{2}}_\epsilon)|^2
\\&\geq(\tr_g\nabla^{f_\epsilon}\delta_{K}^{f_\epsilon}(f^{-\frac{1}{2}}_\epsilon\phi f^{-\frac{1}{2}}_\epsilon),f^{-\frac{1}{2}}_\epsilon\phi f^{-\frac{1}{2}}_\epsilon)
\\&+(f^{-\frac{1}{2}}_\epsilon\phi f^{-\frac{1}{2}}_\epsilon,(\tr_g\nabla^{f_\epsilon}\delta_{K}^{f_\epsilon}(f^{-\frac{1}{2}}_\epsilon\phi f^{-\frac{1}{2}}_\epsilon))^\ast)
\\&=2(\tr_g\nabla^{f_\epsilon}\delta_{K}^{f_\epsilon}(f^{-\frac{1}{2}}_\epsilon\phi f^{-\frac{1}{2}}_\epsilon),f^{-\frac{1}{2}}_\epsilon\phi f^{-\frac{1}{2}}_\epsilon)
\\&=2(\epsilon f^{\frac{1}{2}}_\epsilon\delta\log f_\epsilon(\phi)f^{-\frac{1}{2}}_\epsilon,f^{-\frac{1}{2}}_\epsilon\phi f^{-\frac{1}{2}}_\epsilon)
\\&\geq2\epsilon|f^{-\frac{1}{2}}_\epsilon\phi f^{-\frac{1}{2}}_\epsilon|^2,
\end{split}\end{equation}
where we have used the following pointwise inequality(see \cite[pp.69-70]{LT1995}):
\begin{equation}\begin{split}
(f^{\frac{1}{2}}_\epsilon\delta\log f_\epsilon(\phi)f^{-\frac{1}{2}}_\epsilon,f^{-\frac{1}{2}}_\epsilon\phi f^{-\frac{1}{2}}_\epsilon)\geq|f^{-\frac{1}{2}}_\epsilon\phi f^{-\frac{1}{2}}_\epsilon|^2.
\end{split}\end{equation}
So the maximum principle shows $\phi=0$ and therefore one easily concludes that $\delta_2\hat{L}(\epsilon,f_\epsilon)$ is an isomorphism. By the standard implicit function theorem, we get
\begin{proposition}\label{openness}
$I$ is a nonempty open subset.
\end{proposition}
\begin{lemma}\label{c1estimate}
If $||\log f_\epsilon||_{C^0(M)}$ are uniform bounded, so are $||f_\epsilon||_{C^k(M)}$ for any $k\geq1$.
\end{lemma}
\begin{proof}
In the following, we shall use $C_j(j=1,2,3...)$ to denote uniform constants which may depend on $||\log f_\epsilon||_{C^0(M)}$ and also observe
\begin{equation}\begin{split}
|\psi_{H_\epsilon}|_{H_\epsilon}^2\leq C_1(|\psi_K|^2+|\nabla f_\epsilon|^2),
\end{split}\end{equation}
\begin{equation}\begin{split}
|\nabla f_\epsilon|^2\leq C_2(|\psi_K|^2+|\psi_{H_\epsilon}|^2_{H_\epsilon}).
\end{split}\end{equation}
First of all, we have
\begin{equation}\begin{split}
\Delta_g\tr f_\epsilon
&=\tr_gdK(f^{-1}_\epsilon\delta_{K}f_\epsilon,f_\epsilon)
\\&=(\tr_g\nabla(f^{-1}_\epsilon\delta_{K}f_\epsilon),f_\epsilon)+(f^{-1}_\epsilon\delta_{K}f_\epsilon,\delta_Kf_\epsilon)
\\&=2(\nabla_{\epsilon}^\ast\psi_{H_{\epsilon}}-\nabla_K^\ast\psi_{K},f_\epsilon)
+(f^{-1}_i\delta_K f_\epsilon,\delta_Kf_\epsilon)
\\&\geq-C_3+|\psi_{H_\epsilon}|^2_{H_\epsilon}.
\end{split}\end{equation}
On the other hand, we compute
\begin{equation}\begin{split}
\Delta_g|\psi_{H_\epsilon}|^2_{H_\epsilon}&=-2(\nabla^\ast_{H_\epsilon}\nabla_{H_\epsilon}\psi_{H_\epsilon},\psi_{H_\epsilon})_{H_\epsilon}+2|\nabla_{H_\epsilon}\psi_{H_\epsilon}|^2_{H_\epsilon}
\\&=-2g^{ij}(F_{\nabla_{H_\epsilon}}(\bullet,\frac{\pa}{\pa x^i})(\psi_{H_\epsilon}(\frac{\pa}{\pa x^j})),\psi_{H_\epsilon})_{H_\epsilon}
\\&+2g^{ij}(\psi_{H_\epsilon}(F_{\nabla^{g}}(\bullet,\frac{\pa}{\pa x^i})(\frac{\pa}{\pa x^j})),\psi_{H_\epsilon})_{H_\epsilon}+2|\nabla_{H_\epsilon}\psi_{H_\epsilon}|_{H_\epsilon}^2
\\&-2(\nabla_{H_\epsilon}\nabla_{H_\epsilon}^\ast\psi_{H_\epsilon},\psi_{H_\epsilon})_{H_\epsilon}
-2(D_{H_\epsilon}^\ast D_{H_\epsilon}\psi_{H_\epsilon},\psi_{H_\epsilon})_{H_\epsilon}.
\end{split}\end{equation}
We note the relation
\begin{equation}\begin{split}
\left\{ \begin{array}{ll}
D_{H_\epsilon}\psi_{H_\epsilon}=\frac{1}{2}(F_\nabla+F_\nabla^{\ast H_\epsilon}),\\
F_{\nabla_{H_\epsilon}}+\psi_{H_\epsilon}\wedge\psi_{H_\epsilon}=\frac{1}{2}(F_\nabla-F_\nabla^{\ast H_\epsilon}).
\end{array}\right.
\end{split}\end{equation}
It follows
\begin{equation}\begin{split}
\Delta_g|\psi_{H_\epsilon}|^2_{H_\epsilon}
&=(g^{ij}[[\psi_{H_\epsilon},\psi_{H_\epsilon}(\frac{\pa}{\pa x^i})]-(F_\nabla-F_\nabla^{\ast H_\epsilon})(\bullet,\frac{\pa}{\pa x^i}),\psi_{H_\epsilon}(\frac{\pa}{\pa x^j})],\psi_{H_\epsilon})_{H_\epsilon}
\\&+2(\psi_{H_\epsilon}\circ\Ric_g,\psi_{H_\epsilon})_{H_\epsilon}+2|\nabla_{H_\epsilon}\psi_{H_\epsilon}|_{H_\epsilon}^2
\\&-2\epsilon(\nabla_{H_\epsilon}\log f_\epsilon,\psi_{H_\epsilon})_{H_\epsilon}
-(D_{H_\epsilon}^\ast(F_\nabla+F_\nabla^{\ast H_\epsilon}),\psi_{H_\epsilon})_{H_\epsilon}
\\&\geq|[\psi_{H_\epsilon},\psi_{H_\epsilon}]|^2_{H_\epsilon}-(g^{ij}[(F_\nabla-F_\nabla^{\ast H_\epsilon})(\bullet,\frac{\pa}{\pa x^i}),\psi_{H_\epsilon}(\frac{\pa}{\pa x^j})],\psi_{H_\epsilon})_{H_\epsilon}
\\&-C_4|\psi_{H_\epsilon}|^2_{H_\epsilon}+2|\nabla_{H_\epsilon}\psi_{H_\epsilon}|_{H_\epsilon}^2
\\&+|\nabla_{H_\epsilon}\log f_\epsilon|^2_{H_\epsilon}+|D_{H_\epsilon}^\ast(F_\nabla+F_\nabla^{\ast H_\epsilon})|_{H_\epsilon}|\psi_{H_\epsilon}|_{H_\epsilon}
\\&\geq-C_5-C_6|\psi_{H_\epsilon}|^2_{H_\epsilon}.
\end{split}\end{equation}
Therefore it holds for suitable constants $A$ and $B$,
\begin{equation}\begin{split}
\Delta_g(A\tr f_\epsilon+|\psi_{H_\epsilon}|_{H_\epsilon}^2)
\geq-C_7+B|\psi_{H_\epsilon}|^2_{H_\epsilon},
\end{split}\end{equation}
and thus $|\psi_{H_\epsilon}|^2_{H_\epsilon}(x)\leq C_8$, where $A\tr f_\epsilon+|\psi_{H_\epsilon}|_{H_\epsilon}^2$ attains its maximum at $x$.  This implies
\begin{equation}\begin{split}
\max\limits_{M}|\nabla f_\epsilon|^2
\leq C_9(1+|\psi_{H_\epsilon}|_{H_\epsilon}^2(x))
\leq C_{10}.
\end{split}\end{equation}
\par Now for the Laplician $\Delta_{\nabla_K}=\nabla_K^\ast\nabla_K+\nabla_K\nabla_K^\ast$, we have
\begin{equation}\begin{split}
\Delta_{\nabla_K}f_\epsilon
&=-\frac{1}{2}\tr_g(\nabla\delta_{K}f_\epsilon+\delta_{K}\nabla f_\epsilon)
+\frac{1}{2}\tr_g[\psi_K,[\psi_K,f_\epsilon]].
\end{split}\end{equation}
The elliptic $L^p$-theory indicates
\begin{equation}\begin{split}
||f_\epsilon||_{L_2^p}
&\leq C_{11}(||f_\epsilon||_{L^p}+||\Delta_{\nabla_K}f_\epsilon||_{L^p})
\\&\leq C_{12}(1+||\tr_g\nabla\delta_Kf_\epsilon||_{L^p})
\\&\leq C_{13},
\end{split}\end{equation}
where we have used $\nabla_K^\ast\psi_K+\frac{1}{2}\tr_g(\nabla f_\epsilon^{-1}\delta_Kf_\epsilon+f_\epsilon^{-1}\nabla\delta_Kf_\epsilon)=\epsilon\log f_\epsilon$. Finally, the higher order estimates follow from the elliptic regularity.
\end{proof}
\begin{lemma}\label{c0contradiction}
If $\lim\sup\limits_{\epsilon\rightarrow0}||\log f_\epsilon||_{L^2}=\infty$, there is a nontrivial $\nabla$-invariant sub-bundle of $E$.
\end{lemma}
\begin{proof}
We may assume $l_\epsilon=||\log f_\epsilon||_{L^2}\rightarrow\infty$ as $\epsilon\rightarrow0$. In a straightford way, we compute
\begin{equation}\begin{split}\label{keyinequality}
\Delta_g|\log f_\epsilon|^2
&=2\tr_gdK(\delta_{K}\log f_\epsilon,\log f_\epsilon)
\\&=2\tr_gdK(f^{-1}_\epsilon\delta_{K}f_\epsilon,\log f_\epsilon)
\\&=2(\tr_g\nabla(f^{-1}_\epsilon\delta_{K}f_\epsilon),\log f_\epsilon)
+2(f^{-1}\delta_{K}f_\epsilon,\delta_K\log f_\epsilon)
\\&=4(\nabla_{H_\epsilon}^\ast\psi_{H_\epsilon}-\nabla_K^\ast\psi_K,\log f_\epsilon)+2(f^{-1}\delta_{K}f_\epsilon,\delta_K\log f_\epsilon)
\\&=4(\nabla_{H_\epsilon}^\ast\psi_{H_\epsilon}-\nabla_K^\ast\psi_K,\log f_\epsilon)+2(\Phi[\log f_\epsilon](\nabla\log f_\epsilon),\nabla\log f_\epsilon)
\\&\geq4\epsilon|\log h_\epsilon|^2-4|\nabla_K^\ast\psi_K||\log h_\epsilon|,
\end{split}\end{equation}
where $\Phi(x,y)=(y-x)^{-1}(e^{y-x}-1)$ for $x\neq y$, $\Phi(x,x)=1$ and
\begin{equation}\begin{split}
\log f_\epsilon=\lambda_\alpha e_\alpha\otimes e^\alpha,\
\nabla\log f_\epsilon=(\nabla\log f_\epsilon)^\alpha_\beta e_\alpha\otimes e^\beta,
\end{split}\end{equation}
\begin{equation}\begin{split}
\Phi[\log f_\epsilon](\nabla\log f_\epsilon)=\Phi(\lambda_\beta,\lambda_\alpha)(\nabla\log f_\epsilon)^\alpha_\beta e_\alpha\otimes e^\beta.
\end{split}\end{equation}
Let $\widetilde{\log f_\epsilon}=l_\epsilon^{-1}\log f_\epsilon$ and $(\ref{keyinequality})$ implies
\begin{equation}\begin{split}
|\widetilde{\log f_\epsilon}|\leq l_\epsilon^{-1}(1+l_i)\leq C,
\end{split}\end{equation}
\begin{equation}\begin{split}
\int_Ml_\epsilon(\Phi[l_\epsilon\widetilde{\log f_\epsilon}](\nabla\widetilde{\log f_\epsilon}),\nabla\widetilde{\log f_\epsilon})\dvol_g
&=-2\int_M(\nabla_{H_\epsilon}^\ast\psi_{H_\epsilon}-\nabla_K^\ast\psi_K,\widetilde{\log f_\epsilon})\dvol_g,
\end{split}\end{equation}
where $C$ is a uniform constant independing on $\epsilon$. Since
\begin{equation}\begin{split}
l\Phi(lx,ly)
&=\left\{ \begin{array}{ll}
l,\ x\leq y, \\
\frac{e^{l(y-x)}-1}{y-x},\ x>y,
\end{array}\right.
\longrightarrow
\left\{ \begin{array}{ll}
\infty,\ x\leq y, \\
(x-y)^{-1},\ x>y,
\end{array}\right.
\end{split}\end{equation}
increases monotonically when $l\rightarrow\infty$, it holds \begin{equation}\begin{split}
\int_{M}(\rho[\widetilde{\log f_\epsilon}](\nabla\widetilde{\log f_\epsilon}),\nabla\widetilde{\log f_\epsilon})\dvol_g\leq2\int_{M}(\nabla_K^\ast\psi_K,\widetilde{\log f_\epsilon})\dvol_g.
\end{split}\end{equation}
for any $\rho:\mathbb{R}\times\mathbb{R}\rightarrow\mathbb{R}$ with $\rho(x,y)<(x-y)^{-1}$ whenever $x>y$ and $\epsilon$ is small enough. Due to the zeroth order estimate of $|\widetilde{\log f_\epsilon}|$, by taking small $\rho$, it follows $||\nabla\widetilde{\log f_\epsilon}||_{L_1^2}$ are uniform bounded. We may assume $||u_\infty||_{L^2}=1$ and
\begin{equation}\begin{split}
\widetilde{\log f_\epsilon}\rightarrow u_\infty\in L_1^2S_K,
\end{split}\end{equation}
weakly in $L_1^2$-topology and strongly in $L^2$-topology. Meanwhile, we have
\begin{equation}\begin{split}\label{Simpson}
\int_{M}(\rho[u_\infty](\nabla u_\infty),\nabla u_\infty)\dvol_g\leq2\int_{M}(\nabla_K^\ast\psi_K,u_\infty)\dvol_g<\infty.
\end{split}\end{equation}
\par As long as $(\ref{Simpson})$ is established, by using a similar discussion as \cite[Lemma 5.5]{Si1988} we know the eigenvalues of $u_\infty$ are constants almost everywhere. Since $u_\infty\neq0$ and
\begin{equation}\begin{split}
\int_M\tr u_\infty\dvol_g
&=\lim\limits_{\epsilon\rightarrow0}\int_M\tr\widetilde{\log f_\epsilon}\dvol_g
\\&=\lim\limits_{\epsilon\rightarrow0}\epsilon^{-1}l_\epsilon^{-1}\int_M\tr\nabla^\ast_{H_\epsilon}\psi_{H_\epsilon}\dvol_g
\\&=0,
\end{split}\end{equation}
it admits at least two distinct eigenvalues $\lambda_k$ with $k=1,...,\gamma$ and $2\leq\gamma\leq r$(we may assume $\rank(E)=r>1$). We construct smooth functions $f_s:\mathbb{R}\rightarrow[0,\infty)$, $s=1,2,3,...,\gamma-1$, such that $f_s(x)=1$ if $x\leq\lambda_s$, and $f_s(x)=0$ if $x>\lambda_s$. Set $\Pi_s=f_s[u_\infty]$ which means $f_s$ acts on the eigenvalues of $u_\infty$, it is known that $\Pi_s\neq0,\id_E$ and
\begin{equation}\begin{split}\label{projection}
\Pi_s^\ast=\Pi_s=\Pi_s^2.
\end{split}\end{equation}
Applying $(\ref{Simpson})$ and a similar discussion as \cite[Lemma 5.6]{Si1988}, it yields
\begin{equation}\begin{split}\label{weak1}
0&=(\id_E-\Pi_s)\circ\nabla\Pi_s
\\&=(\id_E-\Pi_s)\circ\nabla\circ\Pi_s
\\&=\nabla(\id_E-\Pi_s)\circ\Pi_s,
\end{split}\end{equation}
and taking adjoint gives
\begin{equation}\begin{split}\label{weak2}
0=&\delta_K\Pi_s\circ(\id_E-\Pi_s)
\\&=-\Pi_s\circ\delta_K\circ(\id_E-\Pi_s)
\\&=\Pi_s\circ\delta_K(\id_E-\Pi_s).
\end{split}\end{equation}
\par Next we show each $\Pi_s$ is smooth. Note Loftin \cite{Lo2009} also considered the similar issue in the study of Donaldson-Uhlenbeck-Yau type theorem for flat vector bundles over special affine manifolds. It is also remarked that we are dealing with the real operator $\nabla$ rather than $\ol\pa$. Indeed, it is purely a local matter, we first employ $(\ref{weak1})$ and $(\ref{weak2})$ to deduce
\begin{equation}\begin{split}\label{Dpi}
\nabla\Pi_s
&=(\id_E-\Pi_s)\circ\nabla\Pi_s+\Pi_s\circ\nabla\Pi_s
\\&=\Pi_s\circ\nabla\Pi_s
\\&=\Pi_s\circ\delta_K\Pi_s+2\Pi\circ[\psi_K,\Pi_s]
\\&=\Pi_s\circ\delta_K(\id_E-\Pi_s)+2\Pi_s\circ[\psi_K,\Pi_s]
\\&=2\Pi_s\circ[\psi_K,\Pi_s],
\end{split}\end{equation}
from which and $u_\infty\in L^\infty(S_K)$ we know $\Pi_s\in L_1^{p}(S_K)$ for any $p>0$. Using
\begin{equation}\begin{split}\label{Dpi2}
\nabla^2\Pi_s
&=2\nabla\Pi_s\circ[\psi_K,\Pi_s]+2\Pi_s\circ\nabla[\psi_K,\Pi_s],
\end{split}\end{equation}
we have $\Pi_s\in L_2^{p}(S_K)$ for any $p>0$. Applying the process repeatedly, we find $\Pi_s\in L_k^{p}(S_K)$ for any $k,p>0$ and thus smooth. 
\par Finally, since each $\Pi_s: E\rightarrow E$ is a nontrivial smooth homomorphisn of bundles and has constant rank, it represents a nontrivial sub-bundle(see \cite[Proposition 2.10]{We2008})
\begin{equation}\begin{split}
F_s\triangleq\im\Pi_s\hookrightarrow E.
\end{split}\end{equation}
Furthermore, $(\ref{weak1})$ indicates $F_s$ is preserved by $\nabla$, a contradiction to the simpleness.
\end{proof}
\begin{proposition}\label{existence}
$I=[0,1]$.
\end{proposition}
\begin{proof}
From $(\ref{keyinequality})$, we see $|\log h_\epsilon|\leq\epsilon^{-1}|\nabla_K^\ast\psi_K|$, Proposition \ref{openness} and Lemma \ref{c1estimate} show $I=(0,1]$. For $\epsilon\rightarrow0$, the simpleness and Lemma \ref{c0contradiction} imply $||\log f_\epsilon||_{L^2}$ are uniformly bounded, from this and $(\ref{keyinequality})$ we know $||\log f_\epsilon||_{C^0(M)}$ must be uniformly bounded. We conclude the a priori estimates up to arbitrary order in view of Lemma \ref{c1estimate} and hence $0\in I$.
\end{proof}
\subsection{Consequences}
A triple $(E,\nabla,H)$ will be called a harmonic bundle if $H$ is harmonic. The semi-simpleness of a harmonic bundle follows from a vector bundle analogue of Gauss-Codazzi equations. In fact, for any $\nabla$-invariant sub-bundle $F$, we write $E=F\oplus F^\perp$ orthogonally and $\nabla$ takes the following form
\begin{equation}\begin{split}
\nabla=\left(\begin{array}{ccc}
\nabla_F & \beta_F \\
0 & \nabla_{F^\perp}
\end{array}\right),
\end{split}\end{equation}
where $\nabla_F$($\nabla_{F^\perp}$) is the induced connection on $F$($F^\perp$) and $\beta_F\in A^1(M,\Hom(F^\perp,F))$ is referred as the second fundamental form. It is easily known that
\begin{equation}\begin{split}
\nabla_H=\left(\begin{array}{ccc}
\nabla_{F,H} & \frac{1}{2}\beta_F \\
-\frac{1}{2}\beta_F^\ast & \nabla_{F^\perp,H}
\end{array}\right),\
\psi_H=\left(\begin{array}{ccc}
\psi_{F,H} & \frac{1}{2}\beta_F \\
\frac{1}{2}\beta_F^\ast & \psi_{F^\perp,H}
\end{array}\right),
\end{split}\end{equation}
where we decompose $\nabla_{F}=\nabla_{F,H}+\psi_{F,H}$ and
$\nabla_{F^\perp}=\nabla_{F^\perp,H}+\psi_{F^\perp,H}$.
We may choose a local normal coordinate and deduce
\begin{equation}\begin{split}
\nabla_H^\ast\psi_H
&=\sum_{i=1}^{\dim M}[\left(\begin{array}{ccc}
\psi_{F,H}(\frac{\pa}{\pa x^i}) & \frac{1}{2}\beta_F(\frac{\pa}{\pa x^i}) \\
\frac{1}{2}\beta_F^\ast(\frac{\pa}{\pa x^i}) & \psi_{F^\perp,H}(\frac{\pa}{\pa x^i})
\end{array}\right), \left(\begin{array}{ccc}
\nabla_{F,H,\frac{\pa}{\pa x^i}} & \frac{1}{2}\beta_F(\frac{\pa}{\pa x^i}) \\
-\frac{1}{2}\beta_F^\ast(\frac{\pa}{\pa x^i}) & \nabla_{F^\perp,H,\frac{\pa}{\pa x^i}}
\end{array}\right)]
\\&=\left(\begin{array}{ccc}
\nabla_{F,H}^\ast\psi_{F,H} & \frac{1}{2}\nabla_{(F^\perp)^\ast\otimes F,H}^\ast\beta_F \\
\frac{1}{2}\nabla_{F^\ast\otimes F^\perp,H}^\ast\beta_F^\ast & \nabla_{F^\perp,H}^\ast\psi_{F^\perp,H}
\end{array}\right)
\\&+\sum_{i=1}^{\dim M}\left(\begin{array}{ccc}
-\frac{1}{2}\beta_F(\frac{\pa}{\pa x^i})\circ\beta_F^\ast(\frac{\pa}{\pa x^i}) & \frac{1}{2}[\psi_{(F^\perp)^\ast\otimes F,H}(\frac{\pa}{\pa x^i}),\beta_F(\frac{\pa}{\pa x^i})]
\\
-\frac{1}{2}[\psi_{F^\ast\otimes F^\perp,H}(\frac{\pa}{\pa x^i}),\beta_F^\ast(\frac{\pa}{\pa x^i})] & \frac{1}{2}\beta_F^\ast(\frac{\pa}{\pa x^i})\circ\beta_F(\frac{\pa}{\pa x^i})
\end{array}\right),
\end{split}\end{equation}
By the harmonicity of $H$, we have
\begin{equation}\begin{split}
2\nabla_{F,H}^\ast\psi_{F,H}=\sum_{i=1}^{\dim M}\beta_F(\frac{\pa}{\pa x^i})\circ\beta_F^\ast(\frac{\pa}{\pa x^i}).
\end{split}\end{equation}
Taking trace on both sides and integrating over $M$ indicate $\beta_F=0$, therefore
\begin{equation}\begin{split}
(E,\nabla,H)=(F,\nabla_F,H)\oplus(F^\perp,\nabla_{F^\perp},H)
\end{split}\end{equation}
splits as harmonic bundles. By an induction argument we conclude that $(E,\nabla)$ is semi-simple and combing this with Proposition \ref{existence} yields
\begin{theorem}\label{semisimpleharmonic}
The existence of harmonic metrics and the semi-simpleness are equivalent.
\end{theorem}
\begin{corollary}\label{correspondence}
Given a Hermitian vector bundle $(E,K)$, we have a natural bijection:
\begin{equation}\begin{split}
I_K:\mathcal{M}_S\rightarrow\mathcal{M}_{I,K}.
\end{split}\end{equation}
\end{corollary}
\begin{proof}
Given a simple connection $\nabla$ with $\nabla f=0$ for $f\in\Gamma(\End(E))$. Let
\begin{equation}\begin{split}
g(\lambda)=(\lambda-\lambda_1)^{m_1}...(\lambda-\lambda_s)^{m_s}
\end{split}\end{equation}
be the characteristic polynomial of $f$, where $\lambda_j$ are the $s$ mutually different eigenvalues of $f$. It is known that $\tr f^k$ must be constants since $d\tr f^k=\tr\nabla f^k=0$. In particular, we conclude $\lambda_j$, $m_j$ are constants and we decompose $E$ into the eigenspaces of $f$:
\begin{equation}\begin{split}
(E,\nabla)=\bigoplus_{i=1}^s(E_{\lambda_i},\nabla|_{E_{\lambda_i}}).
\end{split}\end{equation}
If $s\geq2$ one finds at least one nontrivial $\nabla$-invariant sub-bundle, contradicting to the simpleness. Therefore $f$ only has one eigenvalue, saying $\lambda$. However the simpleness again implies $\ker(f-\lambda\id_E)=E$, $f=\lambda\id_E$ and hence $\nabla$ is irreducible. Furthermore, there is a connection in the $\Aut(E)$-orbit through $\nabla$ such that $K$ is harmonic. In fact, we pick $H$, the harmonic metric given in Theorem \ref{semisimpleharmonic} and it follows for $K^{-1}H=h^{\ast K}h$,
\begin{equation}\begin{split}
h(\nabla)_K^\ast\psi_{h(\nabla),K}=h(\nabla_H^\ast\psi_H)=0.
\end{split}\end{equation}
If $K^{-1}H=\tilde{h}^{\ast K}\tilde{h}$, $\tilde{h}(\nabla)$ is isomorphic to $h(\nabla)$ via $\tilde{h}\circ h^{-1}\in\Unitary(E,K)$. Moreover, let $H_1$, $H_2$ be two harmonic metrics and $h_{12}=H_1^{-1}H_2$, we have
\begin{equation}\begin{split}
\Delta_g\tr h_{12}
=(h_{12}^{-1}\delta_{H_1}h_{12},\delta_{H_1}h_{12})_{H_1}=|h_{12}^{-\frac{1}{2}}\delta_{H_1}h_{12}|_{H_1}^2,
\end{split}\end{equation}
from which and the simpleness we know that $h_{12}$ is a constant mutiple of $\id_E$. Consequently, $\nabla$ uniquely corresponds to an element, saying $I_K(\nabla)=[h(\nabla)]\in\mathcal{M}_{I,K}$. In addition, assuming $I_K(\nabla)=[h_1(\nabla)]$ and $I_K(g(\nabla))=[h_2(g(\nabla))]$, it is not hard to see
\begin{equation}\begin{split}
h_2(g(\nabla))=(h_2\circ g\circ h_1^{-1})(h_1(\nabla)),\
h_2\circ g\circ h_1^{-1}\in\Unitary(E,K).
\end{split}\end{equation}
That is, $I_K$ desends to a bijection from $\mathcal{M}_S$ to $\mathcal{M}_{I,K}$, as required.
\end{proof}
\section{An application in Sasakian geometry}
\subsection{Sasakian manifolds in pseudo-Hermitian geometry}
We start with some basic facts on Sasakian manifolds from the viewpoint of pseudo-Hermitian geometry, more details can be found in \cite{BG2008,DT2006}.  Let $M$ be a $2n+1$-dimensional smooth manifold, a CR structure on $M$ is an integrable rank $n$ complex sub-bundle
\begin{equation}\begin{split}
T_{1,0}M\subseteq T^{\mathbb{C}}M=TM\otimes\mathbb{C}
\end{split}\end{equation}
satisfying $T_{1,0}M\cap T_{0,1}M=\{0\}$ for $T_{0,1}M=\ol{T_{1,0}M}$. We then call $(M,T_{1,0}M)$ a CR manifold and its maximal complex or Levi, distribution is the real rank $2n$ real sub-bundle
\begin{equation}
HM=Re\{T_{1,0}M\oplus T_{0,1}M\}\subseteq TM.
\end{equation}
It carries a complex structure $J:HM\rightarrow HM$ given by
\begin{equation}
J(X+\ol{X})=\sqrt{-1}(X-\ol{X}),\ X\in T_{1,0}M.
\end{equation}
Assume $M$ to be orientable, we define
\begin{equation}
E_x=\{\omega\in T^\ast_xM,HM_x\subseteq\ker\omega\}\subseteq T_x^\ast M,
\end{equation}
then $E$ a real line bundle and any globally defined nowhere vanishing section $\eta$ is called a pseudo-Hermitian structure and the associated Levi form $L_\eta$ is defined by
\begin{equation}\begin{split}
L_\eta(X,\ol{Y})=-\sqrt{-1}d\eta(X,\ol{Y}),\ X,Y\in T_{1,0}M.
\end{split}\end{equation}
An orientable CR manifold $(M,T_{1,0}M)$ is nondegenerate if $L_{\eta}$ is nondegenerate for a pseudo-Hermitian structure $\eta$ and $(M,T_{1,0}M,\eta)$ is called a pseudo-Hermitian manifold. Moreover, $(M,T_{1,0}M,\eta)$ is said to be strictly pseudoconvex CR manifold if $L_\eta$ is positive definite.
\par For a pseudo-Hermitian manifold $(M,T_{1,0}M,\eta)$, $d\eta$ is nondegenerate on $HM$ and thus there is a unique nonvanishing vector field $\xi$(referred to as the characteristic direction) such that
\begin{equation}\begin{split}
\eta(\xi)=1,\ d\eta(\xi,\bullet)=0,
\end{split}\end{equation}
and $\xi$ is transverse to the Levi distribution(that is $TM=HM\oplus\mathbb{R}\xi$). Define a bilinear form
\begin{equation}\begin{split}
G_{\eta}(X,Y)=d\eta(X,JY),\ X,Y\in HM,
\end{split}\end{equation}
the integrability of $T_{1,0}M$ implies that $G_{\eta}$ is $J$-invariant and thus symmetric. One may extend $G_\eta$ to a semi-Riemannian(Riemannian if strictly pseudoconvex CR) metric via
\begin{equation}\begin{split}
g_\eta(X,Y)=G_\eta(X,Y),\ g_\eta(X,\xi)=0,\ g_\eta(\xi,\xi)=1,X,Y\in HM,
\end{split}\end{equation}
and we call it the Webster metric.
\begin{proposition}[\cite{Ta1975,We1978}, see also \cite{DT2006}]\label{TW}
For a pseudo-Hermitian manifold $(M,T_{1,0}M,\eta)$, we extend $J$ to an endomorphism of $TM$ by requiring $J\xi=0$. There exists a unique affine connection $\nabla^{TM}$(called Tanaka-Webster connection) on $TM$ such that
\begin{enumerate}
\item $HM$ is parallel with respect to $\nabla^{TM}$,
\item $\nabla^{TM}g_\eta=0$, $\nabla^{TM}J$=0 and $\nabla^{TM}\xi=0$,
\item The torsion $T_{\nabla^{TM}}$ is pure:
\begin{enumerate}
\item $\tau_{\nabla^{TM}}\circ J+J\circ\tau_{\nabla^{TM}}=0$,
\item $T_{\nabla^{TM}}(X,Y)=0$, $T_{\nabla^{TM}}(X,\ol{Y})=d\eta(X,\ol{Y})\xi$,
\end{enumerate}
where $\tau_{\nabla^{TM}}=T_{\nabla^{TM}}(\xi,\bullet)$(called the pseudo-Hermitian torsion) and $X,Y\in T_{1,0}M$.
\end{enumerate}
\end{proposition}
As an endomorphism, it is known that 
\begin{equation}\begin{split}
\tr\tau_{\nabla^{TM}}=0,\ \tau_{\nabla^{TM}}(T_{1,0}M)\subseteq T_{0,1}M.
\end{split}\end{equation}
On the other hand, the relation between the Levi-Civita connection $\nabla^{g_\eta}$ and the Tanaka-Webster connection $\nabla^{TM}$ is given by(see \cite{DT2006})
\begin{equation}\begin{split}\label{Riemanniantanakedifference}
\nabla^{g_\eta}_XY-\nabla^{TM}_XY&=\left(g_\eta(X,J(Y))-g_\eta(\tau_{\nabla^{TM}}(X),Y)\right)\xi
\\&+\tau_{\nabla^{TM}}(X)\eta(Y)+\eta(X)J(Y)+\eta(Y)J(X),
\end{split}\end{equation}
for two vector fields $X,Y$. Define a vector field $U=\tr_{g_\eta}(\nabla^{g_\eta}-\nabla^{TM})$ and it follows
\begin{lemma}\label{U0}
For a strictly pseudoconvex CR manifold $(M,T_{1,0}M,\eta)$, we have $U=0$.
\end{lemma}
\begin{defn}
A strictly pseudoconvex CR manifold $(M,T_{1,0}M,\eta)$ with vanishing pseudo-Hermitian torsion $\tau_{\nabla^{TM}}$ is called a Sasakian manifold.
\end{defn}
\subsection{A vanishing result}
In what follows we always assume $\{e_1=\xi,e_2,...,e_{2n+1}\}$ is a local frame of $TM$ and $\{e^1=\eta,e^2,...,e^{2n+1}\}$ is a local frame of $T^\ast M$. For a vector bundle $E$ over $M$ with a connection $\nabla$ and a metric $K$ on $E$, we shall use $\nabla_K$ to denote the connection on $\Lambda^\bullet T^\ast M\otimes E$ which acts on $\Lambda^\bullet T^\ast M$ by $\nabla^{TM}$ and acts on $E$ by $\nabla_K$. 
We take the notation
\begin{equation}\begin{split}
\nabla_{H,X,Y}^2\omega=\nabla_{H,X}\nabla_{H,Y}\omega-\nabla_{H,\nabla^{TM}_XY}\omega,
\end{split}\end{equation}
for two vector fields $X,Y$ and $\omega\in A^\bullet(E)$.
\par A connection $\nabla$ on a complex vector bundle $E$ is called projective flat if the curvature satisfies $\sqrt{-1}F_\nabla=\alpha\otimes\id_E$ for a $2$-form $\alpha$.
\begin{lemma}\label{basicharmonic1}
Let $(E,\nabla)$ be a projective flat complex vector bundle over a compact Sasakian manifold $(M,T_{1,0}M,\eta)$ and $H$ be a metric on $E$ with $\nabla_{H,\xi}\nabla_H^\ast\psi_H=0$, then $\nabla_H(\psi_H(\xi))=0$.
\end{lemma}
\begin{proof}
We first note Lemma \ref{U0} implies
\begin{equation}\begin{split}\label{new2}
\nabla_H^\ast\psi_H
&=g_\eta^{ij}\nabla_{H,e_i}(\psi_H(e_j))-g_\eta^{ij}\psi_H(\nabla^{g}_{e_i}e_j)
\\&=g_\eta^{ij}\nabla_{H,e_i}(\psi_H(e_j))-g_\eta^{ij}\psi_H(\nabla^{TM}_{e_i}e_j)
\end{split}\end{equation}
The projective flatness means $D_H\psi_H=0,\ \sqrt{-1}(F_{\nabla_H}+\psi_H\wedge\psi_H)=\alpha\otimes\id_E$ for $\alpha\in A^2(M)$. Next we deduce the symmetry
\begin{equation}\begin{split}\label{new1}
\nabla_{H,e_i}(\psi_H(\xi))
&=\nabla_{H,e_i}\psi_H(\xi)+\psi_H(\nabla^{TM}_{e_i}\xi)
\\&=\nabla_{H,\xi}\psi_H(e_i)-\psi_H(\tau_{\nabla^{TM}}(e_i))-D_H\psi_H(\xi,e_i)
\\&=\nabla_{H,\xi}(\psi_H(e_i)),
\end{split}\end{equation}
where we haved used $\nabla^{TM}\xi=0$ and the vanishing of pseudo-Hermitian torsion. Using these properties again, $(\ref{new2})$ and $(\ref{new1})$ repeatedly, we have
\begin{equation}\begin{split}
\tr_{g_\eta}\nabla_H^2(\psi_H(\xi))
&=g_{\eta}^{ij}\nabla_{H,e_i}\nabla_{H,e_j}(\psi_H(\xi))
-g_{\eta}^{ij}\nabla_{H,\nabla^{TM}_{e_i}e_j}(\psi_H(\xi))
\\&=g_{\eta}^{ij}\nabla^2_{H,e_i,\xi}(\psi_H(e_j))
-g_{\eta}^{ij}\nabla_{H,\nabla^{TM}_{e_i}e_j}(\psi_H(\xi))
\\&=g_{\eta}^{ij}\nabla^2_{H,\xi,e_i}(\psi_H(e_j))
-g_{\eta}^{ij}\nabla_{H,\nabla^{TM}_{e_i}e_j}(\psi_H(\xi))
\\&+g_{\eta}^{ij}\left([F_{\nabla_H}(e_i,\xi),\psi_H(e_j)]
+\nabla_{H,\tau_{\nabla^{TM}}(e_i)}(\psi_H(e_j))\right)
\\&=g_{\eta}^{ij}\nabla^2_{H,\xi,e_i}(\psi_H(e_j))
-g_{\eta}^{ij}\nabla_{H,\nabla^{TM}_{e_i}e_j}(\psi_H(\xi))
\\&-\frac{1}{2}g_{\eta}^{ij}[[\psi_H(e_i),\psi_H(\xi)],\psi_H(e_j)]
\\&=g_\eta^{ij}\nabla_{H,\xi}\nabla_{H,e_i}(\psi_H(e_j))-g_\eta^{ij}\nabla_{H,\xi}(\psi_H(\nabla^{TM}_{e_i}e_j))
\\&-\frac{1}{2}g_{\eta}^{ij}[[\psi_H(e_i),\psi_H(\xi)],\psi_H(e_j)]
\\&=-\nabla_{H,\xi}\nabla_H^\ast\psi_H-\frac{1}{2}g_{\eta}^{ij}[[\psi_H(e_i),\psi_H(\xi)],\psi_H(e_j)].
\end{split}\end{equation}
By Lemma \ref{U0} and the assumption on $H$, it follows
\begin{equation}\begin{split}
\Delta_{g_\eta}|\psi_H(\xi)|^2_H
&=\Div(\nabla^{TM}|\psi_H(\xi)|^2_H)
\\&=2(\tr_{g_\eta}\nabla^2_H(\psi_H(\xi)),\psi_H(\xi))_H+2|\nabla_H(\psi_H(\xi))|^2_H
\\&=-(g_{\eta}^{ij}[[\psi_H(e_i),\psi_H(\xi)],\psi_H(e_j)],\psi_H(\xi))_H^2+2|\nabla_H(\psi_H(\xi))|^2_H
\\&=|[\psi_H(\xi),\psi_H]|_H^2+2|\nabla_H(\psi_H(\xi))|^2_H,
\end{split}\end{equation}
and we thus conclude $\nabla_H(\psi_H(\xi))=0$ by maximum principle.
\end{proof}
Under the setting in Lemma \ref{basicharmonic1}, we define $\mathcal{D}=g_{\eta}^{ij}\left(e^i\wedge\nabla_{H,e_j}-\iota(e_i)\nabla_{H,e_j}\right)$, the Dirac operator on $\Lambda^\ast T^\ast M\otimes\End(E)$. By $(\ref{new2})$ and the projective flatness, we have
\begin{equation}\begin{split}
\mathcal{D}\psi_H
=g_\eta^{ij}e^i\wedge\nabla_{H,e_j}\psi_H+\nabla_H^\ast\psi_H,
\end{split}\end{equation}
and for any $X,Y\in TM$,
\begin{equation}\begin{split}
(g_\eta^{ij}e^i\wedge\nabla_{H,e_j}\psi_H)(X,Y)
&=\nabla_{H,X}\psi_H(Y)-\nabla_{H,Y}\psi_H(X)
\\&=\psi_H(T_{\nabla^{TM}}(X,Y))+D_H\psi_H(X,Y)
\\&=\psi_H(T_{\nabla^{TM}}(X,Y)).
\end{split}\end{equation}
Therefore we have $\mathcal{D}\psi_H=d\eta\otimes\psi_H(\xi)+\nabla_H^\ast\psi_H$ due to Proposition \ref{TW} and the vanishing of pseudo-Hermitian torsion. Furthermore it holds
\begin{equation}\begin{split}
\mathcal{D}^2\psi_H
&=g_\eta^{ij}\left(e^i\wedge\nabla_{H,e_j}(d\eta\otimes\psi_H(\xi))
-(\nabla_{H,e_j}(d\eta\otimes\psi_H(\xi)))(e_i)\right)
\\&=g_\eta^{ij}\left(e^i\wedge d\eta\otimes\nabla_{H,e_j}(\psi_H(\xi))
-d\eta(e_i,\bullet)\otimes\nabla_{H,e_j}(\psi_H(\xi))\right),
\end{split}\end{equation}
where we have used $\nabla^{TM}d\eta=0$. Thanks to Lemma \ref{basicharmonic1}, we find $\mathcal{D}\psi_H=0$ since $\mathcal{D}$ is formal self-adjoint(see \cite[Proposition 2.1]{Pe2002}). Hence, we have the following vanishing result which is essential to construct basic Higgs structures.
\begin{proposition}\label{basicharmonic2}
Let $(E,\nabla)$ be a projective flat complex vector bundle over a compact Sasakian manifold $(M,T_{1,0}M,\eta)$ and $H$ be a metric on $E$ with $\nabla_{H,\xi}\nabla_H^\ast\psi_H=0$, then $H$ is exactly harmonic and satisfies $\psi_H(\xi)=0$.
\end{proposition}
\subsection{Correspondence between projective flat bundles and Higgs bundles}
Below we assume the base space is a compact Sasakian manifold $(M,T_{1,0}M,\eta)$ and denote by $A_{B_{\mathcal{F}_\xi}}^\ast(M)$ the space of basic forms, meaning that any $\omega\in A_{B_{\mathcal{F}_\xi}}^\ast(M)$ satisfies $\iota_\xi\omega=0$ and $L_\xi\omega=0$. Since $\ker\eta\otimes\mathbb{C}=T_{1,0}M\oplus T_{0,1}M$, we similar have the decomposition
\begin{equation}\begin{split}
A_{B_{\mathcal{F}_\xi}}^k(M)\otimes\mathbb{C}=\mathop{\oplus}\limits_{p+q=k}A_{B_{\mathcal{F}_\xi}}^{p,q}(M),
\end{split}\end{equation}
where $A_{B_{\mathcal{F}_\xi}}^{p,q}(M)$ is the sub-space of basic forms of type $(p,q)$. It is easy to see the exterior differential preserves basic forms and hence $d=\pa_\xi+\ol\pa_\xi$ on $A_{B_{\mathcal{F}_\xi}}^k(M)\otimes\mathbb{C}$, such that
\begin{equation}\begin{split}
\pa_\xi: A_{\mathcal{F}_\xi}^{p,q}(M)\rightarrow A_{\mathcal{F}_\xi}^{p+1,q}(M),\
\ol\pa_\xi: A_{\mathcal{F}_\xi}^{p,q}(M)\rightarrow A_{\mathcal{F}_\xi}^{p,q+1}(M).
\end{split}\end{equation}
In the presence of the CR structure $T_{1,0}M$ and the complex structure $J$, there is a transverse holomorphic structure on $M$(with the transverse K\"{a}hler structure $d\eta$) as well as the transverse Hodge theory on the foliated manifold $(M,\mathcal{F}_\xi)$, see \cite[Chapter 7]{To1997} and \cite[pp.275]{BK2021a}.
\par Let $E$ be a complex vector bundle over $M$, it is called transverse holomorphic if it can be trivialized by an open cover $M=\mathop{\cup}\limits_{\alpha}U_\alpha$ so that each transition function $f_{\alpha\beta}$ is basic and holomorphic(that is, $\ol\pa_\xi f_{\alpha\beta}=0$). We have the canonical Dolbeault operator
\begin{equation}\begin{split}
\ol\pa_E: A_{\mathcal{F}_\xi}^{p,q}(M.E)\rightarrow A_{\mathcal{F}_\xi}^{p,q+1}(M,E),
\end{split}\end{equation}
which satisfies $\ol\pa_E^2=0$ and for $\alpha\in A^{p,q}_{B_{\mathcal{F}_\xi}}(M)$, $u\in A^{0}_{B_{\mathcal{F}_\xi}}(M,E)$,
\begin{equation}\begin{split}
\ol\pa_E(\alpha\wedge u)=\ol\pa_{\xi}\alpha\otimes u+(-1)^{p+q}\alpha\wedge\ol\pa_Eu.
\end{split}\end{equation}
A basic Higgs bundle $(E,\theta)$ consists of a transverse holomorphic bundle $E$ and a Higgs field $\theta\in A_{B_{\mathcal{F}_\xi}}^{1,0}(M,\End(E))$ satisfying $\ol\pa_E\theta=0$, $\theta\wedge\theta=0$. Assume further $E$ admits a basic Hermitian metric $H$, one may take all nontrivial sub-Higgs sheaves to define the notion of poly-stability on $(E,\theta)$ in the manner of \cite[Section 3.3]{BH2022}. We say a connection $\nabla$ on a vector bundle $E$ is basic if the associated $D$ restricts to a homomorphism
\begin{equation}\begin{split}
D: A_{\mathcal{F}_\xi}^k(M,E)\rightarrow A_{\mathcal{F}_\xi}^{k+1}(M,E),
\end{split}\end{equation}
and the pair $(E,\nabla)$ is said to be a basic vector bundle. For a basic connection $\nabla$ on $E$, we may define $c_{k,B_{\mathcal{F}_\xi}}(E,\nabla)\in A_{\mathcal{F}_\xi}^{2k}(M,E)$ by
\begin{equation}\begin{split}
\det(I+\frac{\sqrt{-1}}{2\pi}F_\nabla)=\sum\limits_{k=0}^nc_{k,B_{\mathcal{F}_\xi}}(E,\nabla).
\end{split}\end{equation}
For each $k$, the class $c_{k,B_{\mathcal{F}_\xi}}(E,\nabla)$ in $H^{\ast}_{B_{\mathcal{F}_\xi}}(M)$, the cohomology of complex
\begin{equation}\begin{split}
(A^\ast_{B_{\mathcal{F}_\xi}}(M),d|_{A^\ast_{B_{\mathcal{F}_\xi}}(M)}),
\end{split}\end{equation}
is independent of the choice of basic connections and we just write $c_{k,B_{\mathcal{F}_\xi}}(E)$.
\begin{proposition}\label{HarmonicHiggs}
Let $(M,T_{1,0}M,\eta)$ be a compact Sasakian manifold of dimension $2n+1$ and $(E,\nabla)$ be a rank $r$ semi-simple basic projective flat complex vector bundle over $M$, then it determines a poly-stable basic Higgs bundle $(E,\theta)$ over $M$ such that
\begin{equation}\begin{split}\label{Chernnumber}
\int_M\left(c_{2,B_{\mathcal{F}_\xi}}(E)-\frac{r-1}{2r}c_{1,B_{\mathcal{F}_\xi}}^2(E)\right)\wedge(d\eta)^{n-2}\wedge\eta=0.
\end{split}\end{equation}
\end{proposition}
\begin{proof}
We employ Theorem \ref{thm1} and Proposition \ref{basicharmonic2} to find a harmonic metric $H$ on $(E,\nabla)$ satisfying $\psi_H(\xi)=0$. It follows that $\nabla_H$ induces a homomorphism
\begin{equation}\begin{split}
D_H: A_{\mathcal{F}_\xi}^k(M,E)\rightarrow A_{\mathcal{F}_\xi}^{k+1}(M,E).
\end{split}\end{equation}
We decompose $D_H=\pa_{H,\xi}+\ol\pa_{H,\xi}$ and $\psi_H=\theta_{H,\xi}+\ol\theta_{H,\xi}$ such that
\begin{equation}\begin{split}
\pa_{H,\xi}: A_{\mathcal{F}_\xi}^{p,q}(M,E)\rightarrow A_{\mathcal{F}_\xi}^{p+1,q}(M,E),\
\ol\pa_{H,\xi}: A_{\mathcal{F}_\xi}^{p,q}(M,E)\rightarrow A_{\mathcal{F}_\xi}^{p,q+1}(M,E),
\end{split}\end{equation}
and $\theta_{H,\xi}\in A_{\mathcal{F}_\xi}^{1,0}(M,\End(E))$, $\ol\theta_{H,\xi}\in A_{\mathcal{F}_\xi}^{0,1}(M,\End(E))$.
\par The projective flatness implies
\begin{equation}\begin{split}
\pa_{H,\xi}\theta_{H,\xi}=0,\
\ol\pa_{H,\xi}\ol\theta_{H,\xi}=0,\
\pa_{H,\xi}\ol\theta_{H,\xi}+\ol\pa_{H,\xi}\theta_{H,\xi}=0,\
\end{split}\end{equation}
\begin{equation}\begin{split}
\pa_{H,\xi}^2=-\frac{1}{2}[\theta_{H,\xi},\theta_{H,\xi}],\
\ol\pa_{H,\xi}^2=-\frac{1}{2}[\ol\theta_{H,\xi},\ol\theta_{H,\xi}],\
\end{split}\end{equation}
\begin{equation}\begin{split}
[\pa_{H,\xi},\ol\pa_{H,\xi}]=-[\theta_{H,\xi},\ol\theta_{H,\xi}]+\frac{\tr F_\nabla}{r}\otimes\id_E.
\end{split}\end{equation}
Set $D^{''}_{H,\xi}=\ol\pa_{H,\xi}+\theta_{H,\xi}$ and $G_{H,\xi}=D^{''}_{H,\xi}\circ D^{''}_{H,\xi}$, it follows
\begin{equation}\begin{split}\label{sec426}
\tr(G_{H,\xi}\wedge G_{H,\xi})\wedge(d\eta)^{n-2}
&=-2\tr(\ol\theta_{H,\xi}\wedge\ol\theta_{H,\xi}\wedge\theta_{H,\xi}\wedge\theta_{H,\xi})\wedge(d\eta)^{n-2}
\\&-\tr(\ol\pa_{H,\xi}\theta_{H,\xi}\wedge\pa_{H,\xi}\ol\theta_{H,\xi})\wedge(d\eta)^{n-2}
\\&=-2\tr(\ol\theta_{H,\xi}\wedge\ol\theta_{H,\xi}\wedge\theta_{H,\xi}\wedge\theta_{H,\xi})\wedge(d\eta)^{n-2}
\\&-\ol\pa\tr(\theta_{H,\xi}\wedge\pa_{H,\xi}\ol\theta_{H,\xi})\wedge(d\eta)^{n-2}
\\&+\tr(\theta_{H,\xi}\wedge[[\theta_{H,\xi},\ol\theta_{H,\xi}]-\frac{\tr F_\nabla}{r}\otimes\id_E,\ol\theta_{H,\xi}])\wedge(d\eta)^{n-2}
\\&=-\ol\pa\left(\tr(\theta_{H,\xi}\wedge\pa_{H,\xi}\ol\theta_{H,\xi})\wedge(d\eta)^{n-2}\right).
\end{split}\end{equation}
The harmonicity of $H$ is equivalent to $\Lambda_\xi G_{H,\xi}=0$, meaning $G_{H,\xi}$ is primitive, where $\Lambda_\xi$ is the adjoint of the multiplication of $d\eta$. Thus
\begin{equation}\begin{split}
\tr(G_{H,\xi}\wedge G_{H,\xi})\wedge(d\eta)^{n-2}
&=2C_1\tr(G_{H,\xi}^{2,0}\wedge\ast_\xi G_{H,\xi}^{0,2})
-C_1\tr(G_{H,\xi}^{1,1}\wedge\ast_\xi G_{H,\xi}^{1,1})
\\&=2C_1\tr(G_{H,\xi}^{2,0}\wedge\ast_\xi(G_{H,\xi}^{2,0})^{\ast H})
+C_1\tr(G_{H,\xi}^{1,1}\wedge\ast_\xi(G_{H,\xi}^{1,1})^{\ast H})
\\&=C_1\tr\left(G_{H,\xi}\wedge\ast_\xi G_{H,\xi}^{\ast H}\right),
\end{split}\end{equation}
for a constant $C_1$, where $\ast_\xi$ is the basic Hodge star operator given by $\ast_\xi\beta=\ast(\eta\wedge\beta)$ for $\beta\in A^\ast_{B_{\mathcal{F}_\xi}}(M)$. Therefore integrating $\tr(G_{H,\xi}\wedge G_{H,\xi})\wedge(d\eta)^{n-2}\wedge\eta$ over $M$ yields $G_{H,\xi}=0$, meaning $(E,\theta)$ is a basic Higgs bundle with canonical Dolbeault operator $\ol\pa_{H,\xi}$. In addition, the calculation in \cite[Proposition 3.4]{Si1988} shows for a constant $C_2$,
\begin{equation}\begin{split}\label{Chernnumber2}
&\int_M\left(c_{2,B_{\mathcal{F}_\xi}}(E)-\frac{r-1}{2r}c_{1,B_{\mathcal{F}_\xi}}^2(E)\right)\wedge(d\eta)^{n-2}\wedge\eta
\\&=C_2\int_M\left(|F_{\nabla_H+\theta_{H,\xi}+\ol\theta_{H,\xi}}^\perp|^2_H-|\Lambda_\xi F_{\nabla_H+\theta_{H,\xi}+\ol\theta_{H,\xi}}^\perp|^2_H\right)\dvol_{g_\eta}
\\&=0,
\end{split}\end{equation}
by the projective flatness. Using the basic metric $H$, the stability of basic Higgs bundle can be well-defined and one easily concludes the poly-stability, see \cite[Theorem 4.7]{BH2022}.
\end{proof}
\begin{proposition}\label{HiggsHarmonic}
Let $(M,T_{1,0}M,\eta)$ be a compact Sasakian manifold of dimension $2n+1$ and $(E,\theta)$ be a rank $r$ poly-stable basic Higgs bundle satisfying $(\ref{Chernnumber})$ over $M$, then it determines a semi-simple basic projective flat complex vector bundle $(E,\nabla)$ over $M$.
\end{proposition}
\begin{proof}
\par The Sasakian analogue of \cite[Theorem 1.1]{Si1988} shows that there exists a basic Hermitian metric $H$ on $(E,\theta)$ such that(see \cite{BK2021a})
\begin{equation}\begin{split}
\Lambda_\xi F_{\nabla}^\perp=0,\ \nabla=\nabla_H+\theta+\theta^{\ast H},
\end{split}\end{equation}
where $\nabla_H$ is the canonical basic Chern connection. Now by $(\ref{Chernnumber2})$ it follows
\begin{equation}\begin{split}
0&=\int_M\left(c_{2,B_{\mathcal{F}_\xi}}(E)-\frac{r-1}{2r}c_{1,B_{\mathcal{F}_\xi}}^2(E)\right)\wedge(d\eta)^{n-2}\wedge\eta
\\&=C_2\int_M\left(|F_\nabla^\perp|^2_H-|\Lambda_\xi F_\nabla^\perp|^2_H\right)\dvol_{g_\eta}
\\&=C_2\int_M|F_\nabla^\perp|^2_H\dvol_{g_\eta}.
\end{split}\end{equation}
This means $(E,\nabla)$ is a semi-simple basic projective flat complex vector bundle.
\end{proof}
Suppsoe that $(F,\nabla)$ is a flat vector bundle with a harmonic metric $H$, over a compact Sasakian manifold $(M,T_{1,0}M,\eta)$. Define $H_{DR,\mathcal{F}_\xi}^0(M,F)$ to be the space of $\nabla$-parallel elements of $A^0_{\mathcal{F}_\xi}(M,F)$ and $H_{Dol,\mathcal{F}_\xi}^0(M,F)$ to be the space of $D_{H}^{0,1}+\psi_H^{1,0}$-parallel elements of $A^0_{\mathcal{F}_\xi}(M,F)$.
\par In view of Proposition \ref{HarmonicHiggs}, Proposition \ref{HiggsHarmonic} and the natural isomorphism
\begin{equation}\begin{split}
H_{DR,\mathcal{F}_\xi}^0(M,F)\cong H_{Dol,\mathcal{F}_\xi}^0(M,F),
\end{split}\end{equation}
via transverse K\"{a}hler identities(just as \cite[Lemma 1.2]{Si1992}), we conclude
\begin{theorem}\label{Sasakiancorrespondence}
For a compact Sasakian manifold $(M,T_{1,0}M,\eta)$ of dimension $2n+1$, there exists an equivalence between the following two categories:
\begin{itemize}
\item Semi-simple basic projective flat complex vector bundles over $M$.
\item Poly-stable basic Higgs bundles satisfying $(\ref{Chernnumber})$ over $M$.
\end{itemize}
\end{theorem}

\end{document}